\documentclass[reqno]{amsart}
\usepackage{amsfonts}
\usepackage{amsmath}
\usepackage{amssymb}
\usepackage{mathrsfs}
\usepackage[colorlinks]{hyperref}
\usepackage{amsmath,tikz}
\usetikzlibrary{matrix}
\usepackage{algorithm}
\usepackage{algpseudocode}  
\usepackage{algorithmicx}
\usepackage{tikz}
\usetikzlibrary{shadows,positioning}
\usepackage{graphicx,lscape}
\usepackage{subcaption}
\graphicspath{}
\DeclareGraphicsExtensions{.pdf,.png,.jpg}

\begin{document}
\title[\hfil Identification of the time-dependent coefficient]
{Numerical Approaches for Identifying the Time-Dependent Potential Coefficient in the Diffusion Equation}

\author[Arshyn Altybay]{Arshyn Altybay}
\address{
Arshyn Altybay: 
\endgraf
Institute of Mathematics and Mathematical Modeling,
\endgraf
 125 Pushkin str., 050010 Almaty, Kazakhstan
\endgraf
and
\endgraf
Department of Mathematics: Analysis, Logic and Discrete Mathematics, 
  \endgraf
Ghent University, Krijgslaan 281, Building S8, B 9000 Ghent, Belgium
  \endgraf
{\it E-mail address} {\rm arshyn.altybay@gmail.com, arshyn.altybay@ugent.be}
}

\author[M. Ruzhansky]{Michael Ruzhansky}
\address{
  Michael Ruzhansky:
  \endgraf
  Department of Mathematics: Analysis, Logic and Discrete Mathematics
  \endgraf
  Ghent University, Krijgslaan 281, Building S8, B 9000 Ghent, Belgium
  \endgraf
  and
  \endgraf
  School of Mathematical Sciences
  \endgraf
  Queen Mary University of London, United Kingdom
  \endgraf
  {\it E-mail address} {\rm michael.ruzhansky@ugent.be}
}

\thanks{}

\begin{abstract}
We address the inverse problem of identifying a time-dependent potential coefficient in a one-dimensional diffusion equation subject to Dirichlet boundary conditions and a nonlocal integral overdetermination constraint reflecting spatially averaged measurements. After establishing well-posedness for the forward problem and deriving an a priori estimate that ensures uniqueness and continuous dependence on the data, we prove existence and uniqueness for the inverse problem. To compute numerically the unknown coefficient, we propose and compare three numerical methods: an integration-based scheme, a Newton–Raphson iterative solver, and a physics-informed neural network (PINN). Numerical experiments on both exact and noisy data demonstrate the accuracy, robustness, and efficiency of each approach.
\end{abstract}

\maketitle
\numberwithin{equation}{section} 
\newtheorem{theorem}{Theorem}[section]
\newtheorem{lemma}[theorem]{Lemma}
\newtheorem{corollary}[theorem]{Corollary}
\newtheorem{remark}[theorem]{Remark}
\newtheorem{definition}[theorem]
{Definition}
\newtheorem{proposition}[theorem]
{Proposition}
\allowdisplaybreaks
\section{Introduction}\label{1}

Let \( \Omega := (0, l) \) be a bounded spatial domain, and let \( T > 0 \) denote a fixed final time. We consider the inverse problem of jointly identifying a time-dependent potential coefficient \( p(t) \) and the corresponding solution \( u(x,t) \) of the following initial-boundary value problem for the one-dimensional diffusion equation:

\begin{equation} \label{1.1}
    u_t(x,t) = u_{xx}(x,t) - p(t)u(x,t) + f(x,t), \quad (x,t) \in Q := \Omega \times (0,T],
\end{equation}
subject to the initial condition
\begin{equation} \label{1.2}
    u(x,0) = \varphi(x), \quad x \in \overline{\Omega},
\end{equation}
and homogeneous Dirichlet boundary conditions
\begin{equation} \label{1.3}
    u(0,t) = u(l,t) = 0, \quad t \in [0,T].
\end{equation}

Here:
\( u(x,t) \) denotes the state variable (e.g. temperature or concentration),
\( f \in C^{2,1}(\overline{Q}) \) is a known source term,
\( \varphi \in C^2(\overline{\Omega}) \) is the known initial condition,
\( p \in L^{\infty}([0,T]) \) is the unknown reaction or potential coefficient.

To identify the unknown function \( p(t) \), we are given an additional nonlocal  integral overdetermination condition of the form
\begin{equation} \label{1.4}
    \int_0^l u(x,t)\,\omega(x)\,dx = g(t), \quad t \in [0,T],
\end{equation}
where \( \omega \in C^2([0,l]) \) is a given non-negative weight function representing spatial sensitivity of measurements, and \( g \in C^1([0,T]) \) is a given measurement function.

\noindent
\textbf{Objective:} Given \( f(x,t) \), \( \varphi(x) \), \( \omega(x) \), and \( g(t) \), determine the pair \( \{p(t), u(x,t)\} \) that satisfies the PDE \eqref{1.1}, initial condition \eqref{1.2}, boundary conditions \eqref{1.3}, and the overdetermination condition \eqref{1.4}.

Such a problem arises in various contexts involving heat transfer, diffusion, or reaction processes in which a source parameter is present. In particular, when $u(x,t)$ denotes the temperature distribution, the problem can be interpreted as a control problem featuring a source term that serves as the control variable~\cite{Kamynin64, Ionkin77, Cannon86}. 

The integral overdetermination condition \eqref{1.4} represents a fundamental class of inverse problems where measurements correspond to spatially integrated quantities. This formulation was described in detail by Prilepko and his co-authors in their books published in 2000~\cite{Prilepko2000}; it models practical scenarios where sensors capture aggregate field values rather than point measurements, such as total pollutant mass in environmental monitoring or integrated heat flux in thermal systems. The weighting function $\omega(x)$ encodes the measurement device's spatial sensitivity, while 
$g(t)$ provides the time-dependent global observation data essential for coefficient identification. Such integral overdetermination conditions have also been extensively investigated in the recent literature~\cite{Grim20,VanBock22,Hendy22,Karel22,Durd24,Suragan25,Durd25,DurdRah25}.   

Several studies have examined the determination of the time-dependent coefficient $p(t)$ in the heat equation~\eqref{1.1} under integral overdetermination constraints $\int_{\Omega}u(x,t)=E(t)$  alongside various forms of nonlocal boundary conditions (see~\cite{Cannon92,Wang92,Dehghan01,Ivanchov01,Fatullayev08,Kerimov12,Kanca13,Hazanee13,Hazanee14,Durdiev22,Suragan24}). Moreover, numerical approaches to recovering $p(t)$ from integral data have also been investigated (see~\cite{Cannon92,Kanca13,Hazanee13,Hazanee14,Durdiev22}).

In the study conducted by Cannon et al.~\cite{Cannon92}, the inverse problem of identifying a time-dependent source parameter in a semilinear parabolic equation is addressed. The authors established the existence, uniqueness, and continuous dependence upon data for a global solution pair $(p,u)$. Numerical procedures accompanied by examples are also presented, employing additional conditions in the form of either integral over-specification or point measurements $u(x_0,t) = E(t)$. Their work provided important theoretical foundations for such inverse problems.

Kanca~\cite{Kanca13} considered an inverse problem of determining a time-dependent coefficient of heat capacity together with temperature distribution in a heat equation. The problem involves periodic boundary and integral overdetermination conditions, and the existence and uniqueness of classical solutions were demonstrated. Numerically, the author utilised a Crank-Nicolson finite-difference scheme combined with an iteration method; however, accuracy analyses were not included.

Hazanee et al.~\cite{Hazanee13} investigated an inverse time-dependent source problem for a heat equation subject to non-classical boundary and integral overdetermination conditions. The authors showed the existence, uniqueness, and continuous dependence of classical solutions by employing the generalised Fourier method. Their numerical solution employed the Boundary Element Method (BEM) combined with second-order Tikhonov regularisation and regularisation parameter selection based on the generalised cross-validation (GCV) criterion.

Hazanee and Lesnic~\cite{Hazanee14} explored the identification of a time-dependent blood perfusion coefficient in a bioheat equation. They considered the inverse problem as a heat source problem with nonlocal boundary and integral energy overdetermination conditions. Numerically, they used the Boundary Element Method with second-order Tikhonov regularization, demonstrating the accuracy and stability of their solutions through benchmark examples.

Durdiev et al.~\cite{Durdiev22} provided a comparative analysis of finite difference and Fourier spectral numerical methods applied to the inverse problem of determining an unknown coefficient in a parabolic equation with standard initial and boundary conditions. They used an additional condition involving $u_x(0,t) = h(t)$ and concluded from their numerical experiments that the Fourier spectral method significantly outperformed finite difference schemes regarding accuracy.

A substantial body of work has examined the inverse problem of identifying space-dependent coefficients in the heat equation~\cite{Isakov91,Rundell87,Ramm01,Yang08,Deng08,Ruzhansky19,Zhang24,Zhang23}. Notably, in~\cite{Zhang23}, the authors employ a deep neural network to address an inverse potential problem for a parabolic equation, where the unknown potential is space-dependent and the measurement data is taken at the final time. In addition, numerous investigations have addressed the determination of coefficients that vary in both time and space~\cite{Deng10,Deng11,Trucu11,Fan21,Mishra21,ARST21c,Purohit25}, among many others.

While the aforementioned studies have made significant contributions to the theoretical understanding of inverse problems in parabolic equations, they exhibit several limitations from a numerical analysis perspective. First, these works primarily focused on either homogeneous integral overdetermination conditions or pointwise conditions as their additional constraints. Second, they did not provide a comprehensive numerical analysis of the complete inverse problem, particularly concerning convergence properties and error estimation.

Modern scientific computing increasingly adopts neural networks to solve PDEs. Physics-informed neural networks (PINNs) approximate solutions by optimising parameters through PDE-based loss functions~\cite{Raissi2019}. PINNs now tackle diverse challenges~\cite{Bararnia22,Mata23}.

This study advances the field by developing a computational framework that overcomes existing limitations through two key innovations. First, we introduce a non-homogeneous integral overdetermination condition as a supplementary constraint, vastly broadening the method’s applicability to real-world scenarios. Second, we design enhanced numerical solvers that integrate the classical integration method, an optimized Newton–Raphson iteration, and the PINN approach. This hybrid strategy delivers three major benefits: rigorous convergence guarantees, improved computational efficiency, and enhanced numerical stability over long simulations. By combining these methodological advances with a solid mathematical formulation, our work not only surmounts previous challenges but also preserves the theoretical foundations of prior research. The complete Python implementation and all supporting datasets are publicly available on GitHub: \url{https://github.com/Arshynbek/inverse-problem_u_t-u_xx-p-t-u-f}

This paper is organised as follows. In Section \ref{2}, we introduce the key inequalities and foundational lemmas used throughout our analysis. Section \ref{3} establishes the well-posedness results for both the direct and inverse problems. In Section \ref{4}, we describe a numerical method for the direct problem and provide its stability and convergence analysis. Section \ref{5} details three computational approaches for solving the inverse problem. Finally, Section \ref{6} presents and discusses the numerical experiments.


\section{Preliminaries}\label{2}

In this section, we recall some fundamental definitions, inequalities, lemmas and theorems used throughout the paper.

\begin{lemma}\label{lemma2}(Poincaré’s Inequality {\cite[Prop.~8.13, p.~218]{Brezis2010}})
\[
\int_0^l |u(x)|^2\,dx \le \frac{l^2}{\pi^2} \int_0^l |u_x(x)|^2\,dx, \quad \forall u \in H^1_0(0,l).
\]
\end{lemma}

\begin{lemma}\label{lemma3}(Cauchy–Schwarz Inequality  \cite[p.~624]{Evans2010})
Let \(f\) and \(g\) be measurable functions on a domain \(\Omega\) such that \(f,g\in L^2(\Omega)\).  Then
\[
\left|\int_\Omega f(x)\,g(x)\,\mathrm{d}x\right|
\;\le\;
\Bigl(\int_\Omega |f(x)|^2\,\mathrm{d}x\Bigr)^{1/2}
\;\Bigl(\int_\Omega |g(x)|^2\,\mathrm{d}x\Bigr)^{1/2}.
\]
\end{lemma}


\begin{theorem}\label{thm:schauder}(Schauder’s Fixed-Point Theorem, \cite{Zeidler86})
Let \( X \) be a Banach space and \( K \subset X \) a nonempty, closed, bounded, and convex subset. Suppose that the operator \( A : K \rightarrow K \) is continuous and compact. Then the operator \( A \) has at least one fixed point, i.e., there exists \( x \in K \) such that
\[
A(x) = x.
\]
\end{theorem}
\begin{remark}
An operator \( A : K \rightarrow K \) is called \textit{compact} if it maps bounded subsets of \( K \) into relatively compact subsets of \( K \). 
\end{remark}

\begin{theorem}[Arzelà–Ascoli]
Let \( \mathcal{F} \subset C(\overline{Q}) \), where \( \overline{Q} = \overline{\Omega} \times [0,T] \). If \( \mathcal{F} \) is uniformly bounded and equicontinuous on \( \overline{Q} \), then every sequence in \( \mathcal{F} \) has a uniformly convergent subsequence in \( C(\overline{Q}) \).
\end{theorem}

\begin{definition}[Newton–Raphson Method]
The Newton–Raphson method is an iterative root-finding algorithm for solving nonlinear equations of the form \( F(p) = 0 \). Given an initial guess \( p^{(0)} \), the method updates it according to
\[
p^{(j+1)} = p^{(j)} - \frac{F(p^{(j)})}{F'(p^{(j)})},
\]
where \( F'(p) \) is the derivative (or Jacobian in higher dimensions). In the context of inverse problems, it is used to iteratively adjust unknown parameters so that the numerical solution satisfies overdetermined conditions.
\end{definition}

\begin{definition}[Physics-Informed Neural Network \cite{Raissi2019}]
A Physics-Informed Neural Network (PINN) is a machine learning framework that embeds the structure of a differential equation into the training process of a neural network. Instead of relying solely on data, PINNs minimise a loss functional that includes the residuals of the governing PDE, initial/boundary conditions, and any additional constraints. 
\end{definition}

\section{Existence and Uniqueness Results}\label{3}
In this section, we demonstrate the well-posedness of the direct and inverse problems defined by \eqref{1.1}--\eqref{1.4}. 

We have the following assumptions for the problems \eqref{1.1}--\eqref{1.4}:

\begin{enumerate}
    \item  $p \in L^{\infty}([0,T]) $. 
    \item  $\varphi \in C^2([0,l])$ with $\varphi(0) = \varphi(l) = 0$.
    \item  $f \in C^{2,1}(\overline{Q})$ satisfying $f(0,t) = f(l,t) = 0$ for all $t \in [0,T]$.
    \item  $\omega \in C^2([0,l])$ with $\omega(0) = \omega(l) = 0$, and $g \in C^1([0,T])$ with $g(t) \neq 0$ for all $t \in [0,T]$.
\end{enumerate}

\subsection{Existence and uniqueness of the direct problem}
In this subsection, we present existence and uniqueness results as well as an a priori estimate of the direct problem \eqref{1.1}--\eqref{1.3}. 
\begin{theorem}[Existence and Uniqueness]\label{thm:direct-wellposed}
Let the assumptions (1)-(3) be satisfied. Then the initial-boundary value problem \eqref{1.1}--\eqref{1.3} admits a unique solution $u \in C ([0,T]; H^2(0,l) \cap H_0^1(0,l))$ with $u_t \in L^2(0,T; L^2(0,l))$.
\end{theorem}

\begin{proof}
\textbf{Existence:} The solution is constructed via Fourier series. Consider the eigenfunctions $\phi_n(x) = \sqrt{\frac{2}{l}} \sin\left(\frac{n\pi x}{l}\right)$ and eigenvalues $\lambda_n = \left(\frac{n\pi}{l}\right)^2$ for $n \in \mathbb{N}$. Expand:
\[
u(x,t) = \sum_{n=1}^{\infty} u_n(t) \phi_n(x), \quad
\varphi(x) = \sum_{n=1}^{\infty} \varphi_n \phi_n(x), \quad
f(x,t) = \sum_{n=1}^{\infty} f_n(t) \phi_n(x),
\]
where $\varphi_n = \langle \varphi, \phi_n \rangle$, $f_n(t) = \langle f(\cdot,t), \phi_n \rangle$. Substitution into the PDE yields:
\[
u_n'(t) + (\lambda_n + p(t))u_n(t) = f_n(t), \quad u_n(0) = \varphi_n.
\]
Solving this ODE gives:
\begin{equation}\label{sol_coeff}
u_n(t) = \varphi_n e^{-\lambda_nt-\Lambda(t)} + \int_{0}^{t} e^{-\lambda_n(t-s)-(\Lambda(t)-\Lambda(s))}\, f_n(s)\, ds, \quad  \Lambda(t) =  \int_0^t p(s)ds.
\end{equation}

Thus, the full solution is:
\begin{equation}\label{full}
u(t,x) =\sum_{n=1
}^{\infty}\left[ \varphi_n e^{-\lambda_nt-\Lambda(t)} + \int_{0}^{t} e^{-\lambda_n(t-s)-(\Lambda(t)-\Lambda(s))}\, f_n(s)\, ds\right]\phi_n(x).
\end{equation}

Since $p\in L^\infty(0,T)$, let $P:=\|p\|_{L^\infty(0,T)}$. Then, for $0\le s\le t\le T$,
\[
|\Lambda(t)| = \left|\int_0^t p(s)ds \right| \le Pt,\qquad |\Lambda(t)-\Lambda(s)|\le P(t-s),
\]
and hence
\[
\big|e^{-\Lambda(t)}\big|\le e^{Pt},\qquad \big|e^{-(\Lambda(t)-\Lambda(s))}\big|\le e^{P(t-s)}.
\]

The smoothness and boundary vanishing of $\varphi$ and $f(\cdot,t)$ imply the Fourier–sine
coefficient decay
\[
|\varphi_n|=O(n^{-2}),\qquad \sup_{t\in[0,T]}|f_n(t)|=O(n^{-2}).
\]
Consequently, with $\lambda_n\asymp n^2$, there exist $C, C>0$ such that, for all $t\in[0,T]$,
\[
|u_n(t)|\;\le\; C\!\left(\frac{e^{-c n^2 t}}{n^{2}}+\frac{1}{n^{4}}\right),
\qquad
\lambda_n|u_n(t)|\;\le\; C\!\left(e^{-c n^2 t}+\frac{1}{n^{2}}\right).
\]
In particular, for every $\delta\in(0,T]$,
\[
\sup_{t\in[\delta,T]}|u_n(t)|=O(n^{-4}),\qquad
\sup_{t\in[\delta,T]}\lambda_n|u_n(t)|=O(n^{-2}).
\]

Therefore, for each $\delta>0$, the Fourier series defining $u$, $u_x$, and $u_{xx}$
converge \emph{uniformly} on $[0,l]\times[\delta,T]$.
At $t=0$ one has $u(\cdot,0)=\varphi$ in $L^2(0,l)$.
Moreover,
\[
u_t=u_{xx}-p(t)u+f\in L^2\!\big(0,T;L^2(0,l)\big).
\]

\textbf{Uniqueness:} 
Suppose $u_1$ and $u_2$ are two solutions satisfying \eqref{1.1}--\eqref{1.3}. Define $w(x,t) = u_1(x,t) - u_2(x,t)$. Then $w$ satisfies:
\begin{align*}
    w_t &= w_{xx} - p(t) w, \quad (x,t) \in Q, \\
    w(x,0) &= 0, \quad x \in [0,l], \\
    w(0,t) &= w(l,t) = 0, \quad t \in [0,T].
\end{align*}

Consider the energy functional:
\[
E(t) = \frac{1}{2} \int_0^l |w|^2  dx.
\]
Differentiate with respect to $t$ and substitute the PDE:
\begin{align*}
\frac{d}{dt}E(t) &= Re\int_0^l \bar{w}w_t dx \\
&= Re\int_0^l \bar{w} [w_{xx} - p(t) w]  dx \\
&= Re\int_0^l \bar{w} w_{xx}  dx - p(t) \int_0^l |w|^2  dx.
\end{align*}
Integration by parts (using boundary conditions $w(0,t) = w(l,t) = 0$) yields:
\[
\int_0^l \bar{w} w_{xx}  dx = \left[ \bar{w} w_x \right]_0^l - \int_0^l \bar{w_x} w_x   dx = - \int_0^l |w_x|^2 dx.
\]
So,
\[
\frac{d}{dt}E(t) = - \int_0^l |w_x|^2 dx - p(t) \int_0^l |w|^2 dx
\]
Since $p \in L^{\infty}([0,T])$ there exists $P>0$ such that $|p(t)| \leq P\ \text{for almost every  $t \in [0,T]$,}$ Then,
\[
- p(t) \int_0^l |w|^2 dx \leq P \int_0^l |w|^2 dx
\]
Hence,
\[
\frac{d}{dt}E(t) \leq - \int_0^l |w_x|^2 dx + P \int_0^l |w|^2 dx.
\]
Apply Poincaré's inequality for $H_0^1(0,l)$:
\[
\int_0^l |w|^2 dx  \leq \frac{l^2}{\pi^2} \int_0^l | w_x |^2 dx \implies \int_0^l | w_x |^2 dx \geq \frac{\pi^2}{l^2} \int_0^l|w|^2 dx.
\]
Substitute into the inequality:
\[
\frac{d}{dt}E(t) \leq - \frac{\pi^2}{l^2}\int_0^l |w|^2 dx + P \int_0^l |w|^2 dx = (P-\frac{\pi^2}{l^2})\int_0^l |w|^2 dx.
\]
Since $\int_0^l |w|^2 dx = 2E(t)$, we get:
\[
\frac{d}{dt}E(t) \leq 2(P-\frac{\pi^2}{l^2})E(t).
\]
Let $M=2(P-\frac{\pi^2}{l^2})$. Then
\[
\frac{d}{dt}E(t) \leq ME(t).
\]
Since $ E(0) = 0$ and $E(t) \geq 0$, Gronwall’s inequality implies:
\[
E(t) \leq 0, \text{ for all }  t \in [0,T].
\]
Therefore, $E(t)=0$ for all t, so $w(x,t)=0$ almost everywhere. Hence, 
\[
u_1(x,t) = u_2(x,t),
\]
which completes the proof.
\end{proof}

\subsection{A priori estimate for the direct problem}
\begin{theorem}[A priori estimate]\label{thm:apriori}
Let $p \in L^{\infty}([0,T])$, $\varphi \in L^2(0,l)$, 
and $f \in L^2(0,T;L^2(0,l))$.  Then the solution \(u\) of \eqref{1.1}--\eqref{1.3} satisfies:
\begin{equation}\label{eq:apri}
  \| u(\cdot,t) \|_{L^2}^2 \leq e^{(1+2P)t}\|\varphi(\cdot)\|_{L^2}^2 + \int_0^t e^{(1+2P)(t-s)}\| f(\cdot,s) \|_{L^2}^2 ds.
\end{equation}
where $P=\|Re\,p\|_{L^{\infty}}$.
\end{theorem}

\begin{proof}
Multiply \eqref{1.1} by $\bar{u}(x,t)$ and integrate over $(0,l)$:
\[
\int_0^l u_t \, \bar{u} \,\, dx = \int_0^l u_{xx} \,\bar{u} \,\, dx - p(t) \int_0^l |u|^2 \,\, dx + \int_0^l f \,\,\bar{u} \,\, dx.
\]
Integration by parts and boundary conditions yield:
\[
\int_0^l u_{xx} \,\,\bar{u} \, dx = - \| u_x\|_{L^2}^2, \quad \int_0^l u_t \,\ \bar{u} \, dx = \frac{1}{2} \frac{d}{dt} \| u\|_{L^2}^2.
\]
Thus:
\begin{equation} \label{Theo1_eq1}
\frac{1}{2} \frac{d}{dt} \| u \|_{L^2}^2 + \| u_x \|_{L^2}^2 + Re\left(p(t) \| u \|_{L^2}^2 \right) = Re\left(\int_0^l f \,\, \bar{u} \,\, dx\right).
\end{equation}
Using Cauchy-Schwarz and Young's inequalities:
\[
\left| \int_0^l f \,\, \bar{u}  \,\, dx \right| \leq \| f \|_{L^2} \| u \|_{L^2} \leq \frac{1}{2}\| u \|_{L^2}^2 + \frac{1}{2} \| f \|_{L^2}^2.
\]
Since $p \in L^{\infty}([0,T])$, Let $P = \|Re\,p\|_{L^{\infty}}$. Then:
\[
|Re\left(p(t)\right)| \geq -P \implies  Re\left(p(t)\right)\|u\|^2_{L^2} \geq -P\|u\|^2_{L^2}.
\]
Substituting into \eqref{Theo1_eq1} and multiply by 2 we can get
\begin{equation} \label{Theo1_eq2}
 \frac{d}{dt} \| u \|_{L^2}^2 + 2\| u_x \|_{L^2}^2 - 2P\|u\|^2_{L^2} \leq  \| u \|_{L^2}^2 + \| f \|_{L^2}^2.
\end{equation}
Dropping the nonnegative gradient term and changing variable \( t\) by \( s\)  in inequality \eqref{Theo1_eq2} and integrating it over \( s\) from \( 0\) to \(t\) gives: 
\[
  \| u(t) \|_{L^2}^2 \leq \|\varphi\|_{L^2}^2 + (1+2P)\int_0^t \| u(s) \|_{L^2}^2 ds + \int_0^t \| f(s) \|_{L^2}^2 ds.
\]
Applying Grönwall's lemma, we obtain our estimate \eqref{eq:apri}
\end{proof}

\subsection{Existence and uniqueness of the solution of the inverse problem}
\begin{theorem}[Existence for Inverse Problem]\label{thm:inverse}
Under assumptions (1)--(4), the inverse problem \eqref{1.1}--\eqref{1.4} admits at least one solution pair $\{p(t), u(x,t)\}$ with
\[
p \in C([0,T]) \cap L^{\infty}([0,T]), \quad u \in C([0,T];H^2(0,l) \cap H_0^1(0,l)) \text{ and } u_t \in L^2(0,T;L^2(0,l)).
\]
\end{theorem}

\begin{proof}
For any $\hat{p} \in L^{\infty}([0,T])$, Theorem~\ref{thm:direct-wellposed} guarantees a unique solution $u_{\hat{p}} \in  C([0,T];H^2 \cap H_0^1) \text{ and } (u_{\hat{p}})_t \in L^2(0,T;L^2)$ to \eqref{1.1}--\eqref{1.3}.
We employ Schauder's fixed-point theorem (Theorem~\ref{thm:schauder}).
Define the operator $\mathcal{A}: L^{\infty}([0,T]) \to C([0,T])$:
\[
(\mathcal{A}\hat{p})(t) := \frac{ \int_0^l f(x,t)\omega(x)  dx + \int_0^l u_{\hat{p}}(x,t)\omega''(x)  dx - g'(t) }{ g(t) }.
\]
Boundary terms ([$u_x\omega]_0^l$  and $-[u\omega']_0^l$) vanish due to $\omega(0) = \omega(l) = 0$ and $u(0,t)=u(l,t)=0$.

\textbf{Continuity:} Let $\hat{p}_n \to \hat{p}$ in $L^{\infty}([0,T])$(in particular in C). By continuous dependence of solutions (Theorem~\ref{thm:direct-wellposed}), $u_{\hat{p}_n} \to u_{\hat{p}}$ in $C([0,T];L^2)$. Uniform convergence of all terms implies $\mathcal{A}\hat{p}_n \to \mathcal{A}\hat{p}$.

\textbf{Compactness:} Fix $P > 0$ and let $B_R = \left\{\hat{p} \in L^{\infty} : \|\hat{p}\|_{\infty} \leq P \right\}$. By Theorem \ref{thm:apriori}, there exists $C_0$ depending only on $\varphi$, $f$, $l$, and $T$ (but independent of $\hat{p} \in B_R$) such that
\[
\|u_{\hat{p}}(\cdot,t)\|_{L^2} \leq C_0, \text{ for all } t \in [0,T], \quad \|(u_{\hat{p}})_t\|_{L^2(0,T;L^2)} \leq C_1. 
\]
Therefore, for all $\hat{p} \in B_R$,
\[
\|\mathcal{A}\hat{p}(t)\|_{C([0,T])} \leq \frac{1}{\min_{t \in [0,T]} |g(t)|} \left( \|g'\|_\infty + \sup_{t \in [0,T]} \left|\int_0^l f\omega  dx\right| + \|\omega''\|_{L^2} \cdot C_0 \right)  : = M,
\]
and the family $\left \{\mathcal{A}\hat{p}:  \hat{p} \in B_R \right \}$ is equicontinuous and uniformly bounded in $C([0,T])$. So $\mathcal{A}(B_R)$ is relatively compact by Arzelà-Ascoli.

\textbf{Invariance:} Define $K = \{ p \in C([0,T]) : \|p\|_\infty \leq P \}$.
Then $K$ is nonempty, closed, bounded, convex in the Banach space 
$C([0,T])$, and by the bound above, $\mathcal{A}(K) \subseteq K$.

By Schauder's fixed-point theorem, $\mathcal{A}$ has a fixed point $p \in K \subset C([0, T])$. The pair $\{p, u_p\}$ satisfies \eqref{1.1}–\eqref{1.4}.
\end{proof}

\begin{theorem}[Uniqueness of the Solution]\label{thm:uniqueness}
Under assumptions (1)--(4), the solution pair $\{p(t), u(x,t)\}$ to \eqref{1.1}--\eqref{1.4} is unique.
\end{theorem}

\begin{proof}
Suppose $\{p, u\}$ and $\{\tilde{p}, \tilde{u}\}$ are solution pairs. Define
\[
U(x,t) := u(x,t) - \tilde{u}(x,t), \quad P(t) := p(t) - \tilde{p}(t).
\]
Subtracting the two PDEs gives:
\begin{align}
    U_t &= U_{xx} - p(t) U - P(t) \tilde{u}(x,t), \quad (x,t) \in Q, \label{eq:U} \\
    U(x,0) &= 0, \quad x \in [0,l], \nonumber \\
    U(0,t) &= U(l,t) = 0, \quad t \in [0,T], \nonumber
\end{align}
Differentiate \eqref{1.4} for each solution $(u, \tilde{u})$ with respect to $t$:
\[
\int_0^l u_t \,\,\omega \,\, dx = g'(t), \quad\int_0^l \tilde{u}_t \,\, \omega \,\  dx = g'(t),\,\,  t \in [0,T],  
\]
with $g$ from \eqref{1.4}.

Substitute the PDE \eqref{1.1} into these expressions. Using integration by parts and the boundary conditions 
$\omega(0) = \omega(l)=0$, we obtain:
\[
g'(t) = \int_0^l u \omega_{xx} dx - pg(t) + \int_0^l f \omega dx,\\
\]
\[
g'(t) = \int_0^l \tilde{u} \omega_{xx} dx - \tilde{p}g(t) + \int_0^l f \omega dx.\\
\]
Subtracting the two identities yields
\begin{equation} \label{eq:Pg}
P(t) g(t) = \int_0^l U \omega_{xx}  dx.
\end{equation}

Since $g(t) \neq 0$ on $[0,T]$, we can write
\begin{equation} \label{eq:P}
P(t)  = \frac{\int_0^l U \omega_{xx}  dx}{g(t)}.
\end{equation}

Consider the energy functional $E(t) = \frac{1}{2} \int_0^l |U|^2 dx$. Differentiate and substitute \eqref{eq:U}:
\begin{align*}
\frac{d}{dt}E(t) &= Re\int_0^l \bar{U}\,\, U_t  dx \\
&= Re\int_0^l \bar{U} [U_{xx} - p(t)U - P(t)\tilde{u}]  dx \\
&= Re\underbrace{\int_0^l \bar{U} U_{xx}  dx}_{I_1} - p(t)\underbrace{\int_0^l |U|^2  dx}_{I_2} - P(t) Re \underbrace{\int_0^l \bar{U} \tilde{u}  dx}_{I_3}.
\end{align*}

For $I_1$, integrate by parts and apply boundary conditions $U(0,t) = U(l,t) = 0$:
\[
I_1 = \left[ \bar{U} U_x \right]_0^l - \int_0^l |U_x|^2  dx = - \|U_x\|_{L^2}^2.
\]
Thus:
\[
\frac{d}{dt}E(t) = -\|U_x\|_{L^2}^2 - p(t) \|U\|_{L^2}^2 - P(t) Re \int_0^l \bar{U} \tilde{u}  dx.
\]
Since $-\|U_x\|_{L^2}^2 \leq 0$, we have:
\begin{equation}\label{eq:E}
\frac{d}{dt}E(t) \leq -p(t) \|U\|_{L^2}^2 - P(t) Re\int_0^l \bar{U} \tilde{u}  dx.
\end{equation}

From \eqref{eq:P} and the Cauchy–Schwarz inequality:
\[
|P(t)| \leq \frac{1}{|g(t)|} \left| \int_0^l U \omega_{xx}  dx \right| \leq \frac{ \| \omega_{xx} \|_{L^2} }{ |g(t)| } \| U \|_{L^2}.
\]
Define the following constants:
\begin{align*}
M_1 &:= \sup_{t \in [0,T]} \| \tilde{u}(\cdot,t) \|_{L^2} < \infty \quad 
\text{(since $\tilde{u} \in C([0,T];H^2(0,l) \cap H_0^1(0,l))$ by Theorem \ref{thm:direct-wellposed})} \\
M_2 &:= \| \omega_{xx} \|_{L^2} < \infty \quad 
\text{(since $\omega \in C^2([0,l])$ by assumption (4))} \\
m &:= \min_{t \in [0,T]} |g(t)| > 0 \quad 
\text{(since $g \in C^1([0,T])$ and $g(t) \neq 0$ by assumption (4))}.
\end{align*}
Then:
\begin{equation}\label{eq:P(t)}
|P(t)| \leq \frac{M_2}{m}\|U\|_{L^2}.  
\end{equation}
Applying Cauchy-Schwarz to $I_3$:
\begin{equation}\label{eq:M1}
\left| Re\int_0^l \bar{U} \tilde{u}  dx \right| \leq \|U\|_{L^2} \|\tilde{u}\|_{L^2} \leq M_1 \|U\|_{L^2}.
\end{equation}
Substitute \eqref{eq:P(t)} and \eqref{eq:M1} into \eqref{eq:E}:
\begin{align*}
\frac{d}{dt}E(t) &\leq -p(t) \|U\|_{L^2}^2 + |P(t)| M_1 \|U\|_{L^2} \\
&\leq -p(t) \|U\|_{L^2}^2 + \left( \frac{M_2}{m} \|U\|_{L^2} \right) M_1 \|U\|_{L^2} \\
&= \left( -p(t) + \frac{M_1 M_2}{m} \right) \|U\|_{L^2}^2.
\end{align*}
Since $\|U\|_{L^2}^2=2E(t)$, we obtain:
\[
\frac{d}{dt}E(t) \leq 2 \left( -p(t) + \frac{M_1 M_2}{m} \right) E(t).
\]
Let $K := \sup_{t \in [0,T]} \left| -p(t) + \frac{M_1 M_2}{m} \right| < \infty$ (finite because $p \in L^{\infty}([0,T])$ by assumption (1), and $M_1, M_2, m$ are finite with $m > 0$). Then:

\begin{equation} \label{eq:Et}
  \frac{d}{dt}E(t) \leq 2K E(t).  
\end{equation}

Since $E(0) = 0$, Grönwall's inequality applied to \eqref{eq:Et} implies:
\[
E(t) = 0 \quad  \forall t \in [0, T].
\]
Hence, $U = 0$ for all $(x,t) \in \bar{Q}$, so $u=\tilde{u}$.
Substituting $U=0$ into \eqref{eq:P} gives:
\[
P(t)g(t) = 0 \implies P(t) = 0 \quad \forall t \in [0,T],
\]
Since $g(t) \neq 0$, thus, $p = \tilde{p}$, proving uniqueness.
\end{proof}

\section{Numerical solution of the direct problem}\label{4}
In this section, we present a numerical solution of the direct problem for the diffusion equation \eqref{1.1} with initial \eqref{1.2} and boundary \eqref{1.3} using the finite difference method.
We divide the spatial domain $[0,l]$ into $N$ grid points with spacing $h=l/N$ time domain $[0, T]$ into $M$ grid points with spacing $\tau=T/M$. Let the grid points be denoted as  $x_i=ih$, $i=0,1,…, N$, $t_k=k\tau$, for $k=0,1,…, M.$ Let $u_i^k$ be the numerical approximation to $u(x_i, t_k)$. 
To construct a reliable difference scheme for Eq. \eqref{1.1}, we  use the Crank-Nicolson scheme, then problem \eqref{1.1}-\eqref{1.3} can be rewritten in the following form:
\begin{equation}\label{3.1}
\frac{u_i^{k+1} - u_i^{k}}{\tau} - \frac{u_{i+1}^{k+1} - 2u_i^{k+1} + u_{i-1}^{k+1} + u_{i+1}^{k} - 2u_i^{k} + u_{i-1}^{k}}{2h^2} + p^{k+1} u_i^{k+1} = \frac{1}{2}(f_i^{k+1}+f_i^{k}). 
\end{equation}
Next, we multiply both sides by $\tau$ and rearrange as follows:
\begin{equation} \label{3.2}
(1 + \frac{\tau}{h^2} + \tau p^{k+1})u_i^{k+1} - \frac{\tau}{2h^2}(u_{i+1}^{k+1} + u_{i-1}^{k+1}) = (1 - \frac{\tau}{h^2})u_i^k + \frac{\tau}{2h^2}(u_{i+1}^k + u_{i-1}^k) + \frac{\tau}{2}(f_i^{k+1} + f_i^k).
\end{equation}

The discretization of the conditions \eqref{1.2} and \eqref{1.3} give 
\[
 u_i^0 = \varphi(x_i), \quad i = 0, 1, . . . ,N,
\]
\begin{equation} \label{3.3}
u_0^k = 0, \quad u_N^k = 0,  \quad i = 0, 1, . . . ,N, k = 0, 1, . . . , M.
\end{equation}
\subsection{Stability of the implicit difference schemes}
To justify the proposed algorithm, we will derive estimates of the stability of schemes \eqref{3.1} and \eqref{3.2} concerning the initial data and the right-hand side.

\begin{theorem}\label{theorem2} 
\textit{(Stability of the Finite Difference Scheme).}  
\textit{Suppose $\{u_i^k \mid 0 \leq i \leq N, 0 \leq k \leq M\}$ is the solution of the finite difference scheme (\ref{3.2}) with Dirichlet boundary conditions (\ref{3.3}). Assume the coefficient $p^{k+1} > 0 $ \text{for all } k.}
\textit{Then, the scheme is \textbf{unconditionally stable} in the maximum norm ($L^\infty$), meaning for any time step $\tau > 0$ and mesh size $h > 0$, the solution satisfies the stability estimate:}
\[
\|u^{k+1}\|_\infty \leq \|u^k\|_\infty + \tau \|f^{k+1/2}\|_\infty, \quad 0 \leq k \leq M-1,
\]
\textit{where $\|u^k\|_\infty = \max_{0 \leq i \leq N} |u_i^k|$ and $\|f^{k+1/2}\|_\infty = \max_{0 \leq i \leq N} \left|\frac{1}{2}(f_i^{k+1} + f_i^k)\right|$.}\\
\textit{\textbf{Corollary} (Stability for Homogeneous Case). If $f \equiv 0$, the solution is non-increasing in time:}
\[
\|u^{k+1}\|_\infty \leq \|u^k\|_\infty.
\]
\end{theorem}

\begin{proof}\label{proof2} At interior points \(i=1,…,N-1\), the scheme is:  
\[
(1 + \frac{\tau}{h^2} + \tau p^{k+1})u_i^{k+1} - \frac{\tau}{2h^2}(u_{i+1}^{k+1} + u_{i-1}^{k+1}) = (1 - \frac{\tau}{h^2})u_i^k + \frac{\tau}{2h^2}(u_{i+1}^k + u_{i-1}^k) + \tau f_i^{k+1/2},
\] 
where \[ f_i^{k+1/2} = \frac{1}{2}(f_i^{k+1} + f_i^k).\]
By \eqref{3.3}, \(u_0^k = u_N^k = 0\) for all $k$. Thus, the maximum norm $\|u^{k+1}\|_\infty = \max_{0 \leq i \leq N} |u_i^{k+1}|$
 is attained at an interior point $j$

Let \(\|u^{k+1}\|_\infty = |u_j^{k+1}|\). At $i = j$, take absolute values and apply the triangle inequality:
\[
(1 + \frac{\tau}{h^2} + \tau p^{k+1})|u_j^{k+1}| - \frac{\tau}{2h^2}(|u_{j+1}^{k+1}| + |u_{j-1}^{k+1}|) \leq (1 - \frac{\tau}{h^2})|u_j^k| + \frac{\tau}{2h^2}(|u_{j+1}^k| + |u_{j-1}^k|) + \tau|f_j^{k+1/2}|.
\]
Since \(|u_{j\pm1}^{k+1}| \leq \|u^{k+1}\|_\infty = |u_j^{k+1}|\) and \(|u_{j\pm1}^k| \leq \|u^k\|_\infty = |u_j^k|\), we have:
\[
(1 + \frac{\tau}{h^2} + \tau p^{k+1})|u_j^{k+1}| - \frac{\tau}{h^2}|u_j^{k+1}| \leq (1 - \frac{\tau}{h^2})|u_j^k| + \frac{\tau}{h^2}|u_j^k| + \tau|f_j^{k+1/2}|.
\]
We can simplify this as
\[
(1 + \tau p^{k+1})|u_j^{k+1}| \leq |u_j^k| + \tau |f_j^{k+1/2}|.
\]
Since $p^{k+1} > 0$,  $\frac{1}{(1 + \tau p^{k+1})}\leq 1$ and $\frac{\tau}{1 + \tau p^{k+1}} \leq \tau$, we have:
\[
|u_j^{k+1}| \leq |u_j^k| +\tau|f_j^{k+1/2}|.
\]
Taking the maximum over all \(i\):
\[
\|u^{k+1}\|_\infty \leq \|u^k\|_\infty + \tau \|f^{k+1/2}\|_\infty.
\]
For homogeneous cases  $f \equiv 0$, the solution satisfies:
\[
\|u^{k+1}\|_\infty \leq \|u^k\|_\infty,
\]
which shows that the scheme is unconditionally stable. 
\end{proof}

\subsection{Convergence Analysis}
We are now ready to show the convergence of the difference scheme \eqref{3.2}.
\begin{theorem}\label{theorem3}
\textbf{(Convergence).} \textit{Let \(u(x,t)\) be the exact solution of the PDE with sufficient smoothness, and \(\{u_i^k\}\) be the numerical solution of the finite difference scheme \eqref{3.2}. Assume the truncation error is \(\mathcal{O}(\tau^2 + h^2)\). Then, under the stability result of Theorem \ref{theorem2}, the numerical solution converges to the exact solution with the error bound:}
\[
\|e^{k}\|_\infty \leq C \left(\tau^2 + h^2\right), \quad 0 \leq k \leq M,
\]
\textit{where \(e_i^k = u(x_i, t_k) - u_i^k\), and \(C > 0\) is a constant independent of \(\tau\) and \(h\).}
\end{theorem} 

\begin{proof} Let define the error \(e_i^k = u(x_i, t_k) - u_i^k\). Substituting \(u(x_i, t_k)\) into the scheme \eqref{3.2}, the error satisfies:
\[
\frac{e_i^{k+1} - e_i^{k}}{\tau} - \frac{e_{i+1}^{k+1} - 2e_i^{k+1} + e_{i-1}^{k+1} + e_{i+1}^{k} - 2e_i^{k} + e_{i-1}^{k}}{2h^2} + p^{k+1} e_i^{k+1} = T_i^{k+1/2},
\]
where \(T_i^{k+1/2}\) is the truncation error at time \(t_{k+1/2} = t_k + \tau/2\).

Using Taylor expansions, the Crank-Nicolson approximation: \(\frac{u_i^{k+1} - u_i^k}{\tau}\) has a truncation error of \(\mathcal{O}(\tau^2)\). The central difference approximation \\
\[
\frac{u_{i+1}^{k+1} - 2u_i^{k+1} + u_{i-1}^{k+1} + u_{i+1}^k - 2u_i^k + u_{i-1}^k}{2h^2}
\]
has a truncation error of \(\mathcal{O}(h^2)\).
And the reaction and source terms: The  \(p^{k+1} u_i^{k+1}\) and \(\frac{1}{2}(f_i^{k+1} + f_i^k)\) introduce no additional error beyond \(\mathcal{O}(\tau^2 + h^2)\).
Thus, the truncation error satisfies:
\[
\|T^{k+1/2}\|_\infty \leq C_T (\tau^2 + h^2),
\]
where \(C_T > 0\) is a constant.

From Theorem \ref{theorem2} we know  that the error equation inherits the stability of the original scheme:
\[
\|e^{k+1}\|_\infty \leq \|e^k\|_\infty + \tau \|T^{k+1/2}\|_\infty.
\]
Substituting the truncation error bound, we get
\[
\|e^{k+1}\|_\infty \leq \|e^k\|_\infty + \tau C_T (\tau^2 + h^2).
\]
Iterating over time steps yields
\[
\|e^{k}\|_\infty \leq \|e^0\|_\infty + \sum_{m=1}^k \tau C_T (\tau^2 + h^2).
\]
Since \(\|e^0\|_\infty = 0\) (assuming exact initial conditions), and \(\sum_{m=1}^k \tau = t_k \leq T\), we have
\[
\|e^{k}\|_\infty \leq C_T T (\tau^2 + h^2).
\]
Combining results, we obtain
\[
\|e^{k}\|_\infty \leq C (\tau^2 + h^2),
\]
where \(C = C_T T\). 
\end{proof}

\section{Numerical Solution of the Inverse Problem}\label{5}

This section presents and compares three numerical methods for simultaneously identifying the time-dependent coefficient \( p(t) \) and computing the solution \( u(x,t) \) of the inverse problem \eqref{1.1}–\eqref{1.4}. The first method, referred to as the integration-based approach, is a classical technique relying on analytical manipulation of the overdetermination condition. The second is a Newton–Raphson-based iterative method that refines \( p(t) \) by minimising the residual in the integral condition. Finally, we present a third approach based on the Physics-Informed Neural Network (PINN) framework, which enables mesh-free learning of both \( u(x,t) \) and \( p(t) \) by embedding the governing PDE into the training process.

The three approaches are implemented independently and compared through numerical experiments. All numerical examples are conducted on the domain \( [0,1] \times [0,T] \), and the results are benchmarked against an exact analytical solution to validate accuracy.

\subsection{Integration-based approach for identification of \texorpdfstring{$\{p(t), u(x,t)\}$}{p(t), u(x,t)}}

We first consider a classical approach that leverages the overdetermination condition \eqref{1.4} to derive an explicit formula for the unknown coefficient \( p(t) \), assuming that the solution \( u(x,t) \) is available at each time step via numerical simulation of the direct problem \eqref{1.1}–\eqref{1.3}.

Differentiating the integral condition \eqref{1.4} with respect to time yields:
\begin{equation}\label{eq:5.1}
\frac{d}{dt} \left( \int_0^l u(x,t)\, \omega(x)\, dx \right)
= \int_0^l u_t(x,t)\, \omega(x)\, dx = g'(t),
\end{equation}

Substituting the PDE \eqref{1.1} into \eqref{eq:5.1} gives:
\[
g'(t) = \int_0^l \left( u_{xx}(x,t) - p(t)\, u(x,t) + f(x,t) \right) \omega(x)\, dx.
\]

We now integrate the \( u_{xx} \) term by parts twice:
\[
\int_0^l u_{xx}(x,t)\, \omega(x)\, dx
= \left[ u_x \omega \right]_0^l - \left[ u \omega' \right]_0^l + \int_0^l u(x,t)\, \omega''(x)\, dx.
\]
Since \( u(0,t) = u(l,t) = 0 \) and \( \omega(0) = \omega(l) = 0 \), the second boundary term vanishes, and we obtain:
\[
\int_0^l u_{xx}(x,t)\, \omega(x)\, dx =  \int_0^l u(x,t)\, \omega''(x)\, dx.
\]

Substituting back, and isolating \( p(t) \), we arrive at the reconstruction formula:
\begin{equation}\label{eq:pt_integral}
p(t) = \frac{\int_0^l u(x,t)\, \omega''(x)\, dx + \int_0^l f(x,t)\, \omega(x)\, dx - g'(t) }{ g(t) }.
\end{equation}
Here, we have assumed in assumption (4) that $g(t) \neq 0$.
This expression allows us to recover \( p(t) \) pointwise at each time level after computing the numerical solution \( u(x,t) \) of the direct problem.

\medskip

The full computational procedure is summarised in Algorithm~\ref{alg:integration}, where we use the Crank–Nicolson scheme described in Section~\ref{4} to compute \( u(x,t) \).

\begin{algorithm}[H]
\caption{Integration-based method for identifying \( \{p(t), u(x,t)\} \)}
\label{alg:integration}
\begin{algorithmic}[1]
\Require Initial condition \( u_i^0 = \varphi(x_i) \), boundary values \( u_0^k = u_N^k = 0 \)
\Require Known functions \( f(x,t),\, \omega(x),\, g(t),\, g'(t) \), grid steps \( h, \tau \), and total steps \( N, M \)
\State Initialize \( p^0 \) using formula \eqref{eq:pt_integral}
\For{$k = 0$ to $M-1$}
    \State Solve the linear system from scheme \eqref{3.2} for \( u^{k+1} \)
    \State Set boundary conditions
    \State Update \( p^{k+1} \) using equation \eqref{eq:pt_integral}
\EndFor
\State \Return \( \{ u_i^k \}_{i,k=0}^{N,M} \), \( \{ p^k \}_{k=0}^M \)
\end{algorithmic}
\end{algorithm}

\subsection{Newton--Raphson approach for identification of \texorpdfstring{$\{p(t), u(x,t)\}$}{p(t), u(x,t)}}

We now describe a more accurate and iterative method for identifying the time-dependent coefficient \( p(t) \) based on minimising the discrepancy in the overdetermination condition \eqref{1.4}. At each time step \( t^{k+1} \), the unknown value \( p^{k+1} \) is computed by solving the nonlinear equation
\begin{equation} \label{eq:newton_residual}
F(p^{k+1}) := \sum_{i=1}^{N-1} u_i^{k+1}(p^{k+1})\, \omega(x_i)\, h - g(t^{k+1}) = 0,
\end{equation}
where \( u_i^{k+1}(p^{k+1}) \) denotes the discrete numerical solution of the direct problem at time \( t^{k+1} \) computed using scheme \eqref{3.2} with current guess for \( p^{k+1} \).

We apply the classical Newton–Raphson iteration to solve \eqref{eq:newton_residual}. Given an initial approximation \( p^{k+1,(0)} \), the update rule is:
\begin{equation} \label{eq:newton_update}
p^{k+1,(j+1)} = p^{k+1,(j)} - \frac{F(p^{k+1,(j)})}{F'(p^{k+1,(j)})},
\end{equation}
where \( F'(p) \) is the derivative of the residual function with respect to \( p \), given by
\begin{equation} \label{eq:dfdp}
F'(p^{k+1}) = \sum_{i=1}^{N-1} \frac{\partial u_i^{k+1}}{\partial p^{k+1}} \, \omega(x_i)\, h \, .
\end{equation}

To compute \( \frac{\partial u_i^{k+1}}{\partial p^{k+1}} \), we differentiate the finite difference scheme \eqref{3.2} with respect to \( p^{k+1} \), obtaining the following linear system:
\begin{equation} \label{eq:du_dp_system}
\left( 1 + \frac{\tau}{h^2} + \tau p^{k+1} \right) \frac{\partial u_i^{k+1}}{\partial p^{k+1}} - \frac{\tau}{2h^2} \left( \frac{\partial u_{i+1}^{k+1}}{\partial p^{k+1}} + \frac{\partial u_{i-1}^{k+1}}{\partial p^{k+1}} \right)
= -\tau u_i^{k+1},
\end{equation}
with homogeneous Dirichlet conditions imposed on \( \frac{\partial u_0^{k+1}}{\partial p^{k+1}} \) and \( \frac{\partial u_N^{k+1}}{\partial p^{k+1}} \).

This system has the same tridiagonal structure as the original discretisation, and is efficiently solved using the Thomas algorithm. The full procedure is presented in Algorithm~\ref{alg:newton}.

\begin{algorithm}[H]
\caption{Newton--Raphson method for identifying \( \{p(t), u(x,t)\} \)}
\label{alg:newton}
\begin{algorithmic}[1]
\Require Initial condition \( u_i^0 = \varphi(x_i) \), step sizes \( h, \tau \), tolerance \( \epsilon \)
\Require Known data: \( f(x,t),\, \omega(x),\, g(t) \)
\State Set \( p^0 \gets p_{\text{init}} \)
\For{$k = 0$ to $M-1$}
    \State Initialize iteration: \( p^{k+1,(0)} \gets p^k \), set \( j \gets 0 \)
    \Repeat
        \State Solve Crank--Nicolson scheme \eqref{3.2} with \( p^{k+1,(j)} \) to get \( u_i^{k+1} \)
        \State Evaluate residual \( F(p^{k+1,(j)}) \) using \eqref{eq:newton_residual}
        \State Solve sensitivity system \eqref{eq:du_dp_system} for \( \frac{\partial u_i^{k+1}}{\partial p^{k+1}} \)
        \State Evaluate derivative \( F'(p^{k+1,(j)}) \) via \eqref{eq:dfdp}
        \State Update \( p^{k+1,(j+1)} \gets p^{k+1,(j)} - F / F' \)
        \State \( j \gets j + 1 \)
    \Until{$ |F(p^{k+1,(j)})| < \epsilon $}
    \State Set \( p^{k+1} \gets p^{k+1,(j)} \)
\EndFor
\State \Return \( \{ u_i^k \},\, \{ p^k \} \)
\end{algorithmic}
\end{algorithm}
\begin{remark}
The Newton-Raphson algorithm benefits from several advantageous features. First, both the forward and sensitivity problems share the same tridiagonal matrix structure, allowing for the reuse of efficient solvers such as the Thomas algorithm. Each iteration therefore, has linear complexity \( \mathcal{O}(N) \), where \( N \) is the number of spatial discretisation points, making the method computationally efficient.

Moreover, the implicit time discretisation contributes to the scheme's numerical stability without requiring additional regularisation. The residual \( \|F(p^{k+1})\| \), evaluated at each iteration, serves as both a stopping criterion and a diagnostic tool for monitoring convergence and solution quality.
\end{remark}

\subsection{Physics-Informed Neural Network (PINN) approach for identification of \texorpdfstring{$\{p(t), u(x,t)\}$}{p(t), u(x,t)}}

As a modern alternative to classical numerical schemes, we propose a physics-informed neural network (PINN) approach for the simultaneous identification of the time-dependent coefficient \( p(t) \) and the solution \( u(x,t) \) to the inverse problem \eqref{1.1}--\eqref{1.4}. In this method, the unknowns are approximated by neural networks trained to satisfy the governing PDE, initial and boundary conditions, and the integral overdetermination condition.

\medskip

\textbf{Neural network representation.}
We represent the solution \( u(x,t) \) by a trial function of the form:
\begin{equation} \label{eq:trial}
u_\theta(x,t) = (1 - t)\, \varphi(x) + x(1 - x)\, t\, N_\theta(x,t),
\end{equation}
where \( N_\theta(x,t) \) is a feedforward neural network with parameters \( \theta \). This construction ensures that the initial and boundary conditions,
\[
u_\theta(x,0) = \varphi(x), \quad u_\theta(0,t) = u_\theta(1,t) = 0,
\]
are exactly satisfied for all \( (x,t) \in [0,1] \times [0,T] \).

The unknown coefficient \( p(t) \) is approximated by a second neural network \( p_\eta(t) \), parameterised by weights \( \eta \). The two networks are trained simultaneously using a composite loss functional based on the residuals of the PDE and the integral constraint.

\medskip

\textbf{Loss formulation.}
Let \( \{(x_i, t_i)\}_{i=1}^{N_f} \subset (0,1) \times (0,T) \) be interior collocation points used to enforce the PDE. The residual loss is defined as:
\begin{equation} \label{eq:pinn_pde}
\mathcal{L}_{\mathrm{PDE}} = \frac{1}{N_f} \sum_{i=1}^{N_f}
\left( \frac{\partial u_\theta}{\partial t}(x_i, t_i) - \frac{\partial^2 u_\theta}{\partial x^2}(x_i, t_i)
+ p_\eta(t_i)\, u_\theta(x_i, t_i) - f(x_i, t_i) \right)^2.
\end{equation}

To enforce the integral overdetermination condition \eqref{1.4}, we define the constraint loss:
\begin{equation} \label{eq:pinn_integral}
\mathcal{L}_{\mathrm{int}} = \frac{1}{N_g} \sum_{j=1}^{N_g}
\left( \int_0^1 u_\theta(x, t_j)\, \omega(x)\, dx - g(t_j) \right)^2,
\end{equation}
where \( \{t_j\}_{j=1}^{N_g} \subset [0,T] \) are selected quadrature nodes, and the integral is approximated using a numerical quadrature rule (e.g., trapezoidal or Simpson’s).

The total loss functional is:
\begin{equation} \label{eq:pinn_loss}
\mathcal{L}(\theta, \eta) = \mathcal{L}_{\mathrm{PDE}} + \lambda\, \mathcal{L}_{\mathrm{int}}\,\,,
\end{equation}
where \( \lambda > 0 \) is a penalty parameter controlling the weight of the constraint.

\medskip

\textbf{Neural network architecture.}
The solution network \( N_\theta(x,t) \) consists of 2 hidden layers with 64 neurons each and hyperbolic tangent (tanh) activations. The coefficient network \( p_\eta(t) \) uses 2 hidden layers with 32 neurons each, also using tanh activations. The networks are trained jointly using a two-phase strategy: an initial training phase using the Adam optimiser to explore the solution landscape, followed by a second phase employing the L-BFGS optimiser to refine convergence with high accuracy.

The overall architecture of the PINN—including the trial solution network \( u_\theta(x,t) \), the coefficient network \( p_\eta(t) \), and their respective outputs—is depicted in Figure~\ref{fig:pinn-architecture}.

\begin{figure}[H]
\centering
\includegraphics[width=0.7\textwidth]{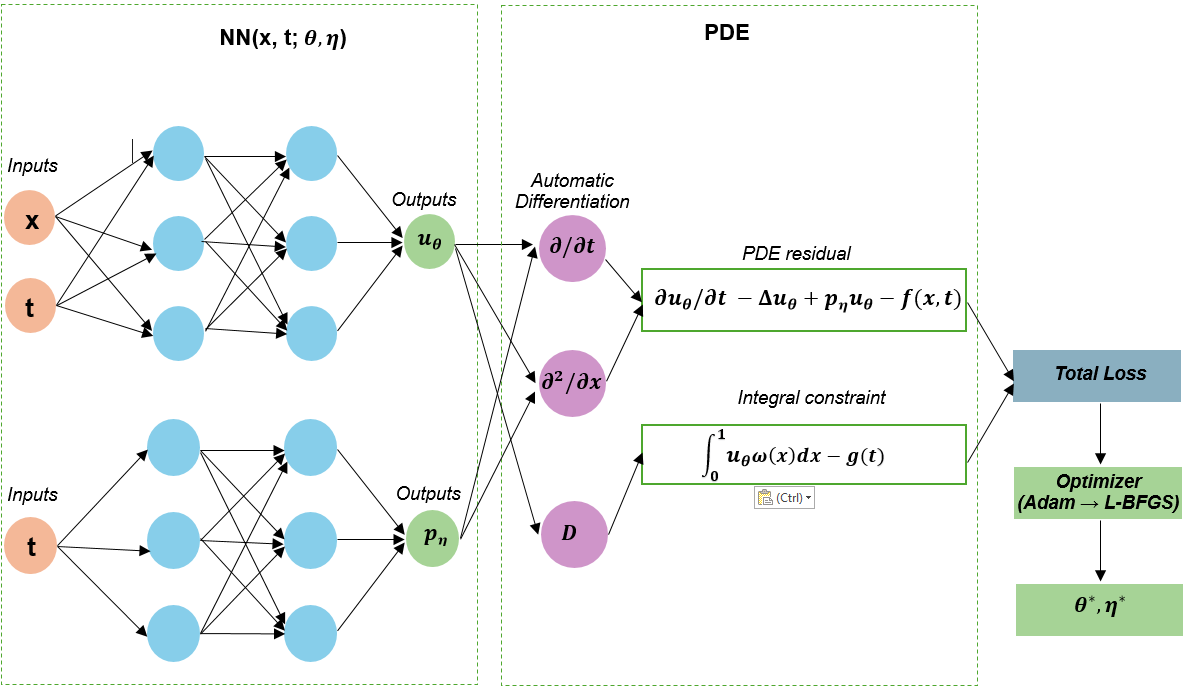}
\caption{Architecture of the PINN for simultaneous approximation of \( u(x,t) \) and \( p(t) \).}
\label{fig:pinn-architecture}
\end{figure}

The training strategy is summarised in Algorithm~\ref{alg:pinn}.

\begin{algorithm}[H]
\caption{PINN for simultaneous identification of \( p(t) \) and \( u(x,t) \)}
\label{alg:pinn}
\begin{algorithmic}[1]
\Require Collocation sets \( \mathcal{S}_{\mathrm{int}} \subset (0,1) \times (0,T) \), \( \mathcal{S}_{\mathrm{intc}} \subset [0,T] \)
\Require Initial function \( \varphi(x) \), known functions \( f(x,t), \omega(x), g(t) \)
\State Define trial solution \( u_\theta(x,t) \) using \eqref{eq:trial}
\State Initialize networks \( N_\theta(x,t) \) and \( p_\eta(t) \)
\For{epoch \( = 1 \) to \( K \)}
    \State Sample mini-batch from \( \mathcal{S}_{\mathrm{int}} \) for evaluating \( \mathcal{L}_{\mathrm{PDE}} \) via \eqref{eq:pinn_pde}
    \State Sample time points \( \{t_j\}_{j=1}^{N_g} \subset \mathcal{S}_{\mathrm{intc}} \), compute \( \mathcal{L}_{\mathrm{int}} \) via \eqref{eq:pinn_integral}
    \State Compute total loss \( \mathcal{L} = \mathcal{L}_{\mathrm{PDE}} + \lambda\, \mathcal{L}_{\mathrm{int}} \)
    \State Update parameters \( \theta, \eta \) using gradient-based optimizer (Adam or L-BFGS)
\EndFor
\State \Return trained networks \( u_\theta(x,t) \), \( p_\eta(t) \)
\end{algorithmic}
\end{algorithm}

\medskip

\textbf{Advantages.}
The PINN framework offers several advantages: it produces smooth approximations of both the solution and the coefficient, handles sparse or noisy data naturally, and is mesh-free. Moreover, it provides a continuous representation of \( u(x,t) \) and \( p(t) \) that can be evaluated at arbitrary points in the domain.

\begin{remark}
The key challenge in PINNs for inverse problems lies in balancing the different terms in the loss function. In our implementation, the penalty parameter \( \lambda \) was chosen empirically to ensure that both the PDE residual and the constraint are well enforced. Additionally, automatic differentiation is used to evaluate all derivatives appearing in the loss terms.
\end{remark}

\section{Numerical Experiments}\label{6}

To validate the proposed methodologies and assess their accuracy and convergence properties, we present a set of numerical experiments. We consider a manufactured solution to the inverse problem \eqref{1.1}--\eqref{1.4}, which allows for direct comparison against known ground truth. Two classical methods--the integration-based approach and the Newton-Raphson-based approach--are compared with the Physics-Informed Neural Network (PINN) strategy. And to assess the robustness of these methods, we tested them using noisy integral data $g(t)$ and its time derivative.

\textbf{Test problem with known solution.}
We consider the inverse problem on the domain \( Q = (0,1) \times (0,1) \), with the exact solution and coefficient given by:
\[
    u(x,t) = e^t \sin(\pi x), \qquad p(t) = e^{-t}.
\]
From this, the source function \( f(x,t) \) and the overdetermination data \( g(t) \) are computed accordingly:
\[
    f(x,t) = \sin(\pi x)(1 + \pi^2 e^t + e^t), \qquad \omega(x) = \sin(\pi x), \qquad g(t) = \frac{1}{2} e^t.
\]

This problem satisfies the assumptions of the theoretical framework and provides a smooth benchmark for verifying accuracy.

\textbf{Error metrics.}
To evaluate the accuracy of the numerical solutions, we use both maximum and $L^2$ norm errors:
\begin{align*}
    \text{Max error in } u(x,t): \quad & Er(u)_{N,M} := \max_{0 \leq i \leq N} |u_i^M - U_i^M|, \\
    \text{Max error in } p(t): \quad & Er(p)_M := \max_{0 \leq k \leq M} |p^k - P^k|, \\
    \text{$L^2$ error in } u(x,t): \quad & E(u)^{(2)}_{N,M} := \sqrt{h \sum_{i=0}^{N} (u_i^M - U_i^M)^2}, \\
    \text{$L^2$ error in } p(t): \quad & E(p)^{(2)}_M := \sqrt{\tau \sum_{k=0}^{M} (p^k - P^k)^2},
\end{align*}
where \( (U_i^M, P^k) \) denote the exact solutions and \( (u_i^M, p^k) \) the numerically computed approximations.

\subsection{Results for the integration-based approach}
Tables~\ref{tab:integration-tau} and~\ref{tab:integration-h} report the accuracy behaviour of the integration-based method.

\begin{table}[H]
\centering
\caption{Integration based approach: varying $\tau$ for fixed $h = 1/100$}
\label{tab:integration-tau}
\begin{tabular}{c|c|c|c|c}
\hline
$\tau$ & $Er(u)$ & $Er(p)$ & $E^{(2)}(u)$ & $E^{(2)}(p)$ \\
\hline
1/200  & 2.76e-3 & 1.05e-2 & 1.95e-3 & 8.76e-3 \\
1/400  & 6.90e-4 & 2.63e-3 & 4.88e-4 & 2.19e-3 \\
1/800  & 1.73e-4 & 6.57e-4 & 1.22e-4 & 5.48e-4 \\
1/1600 & 4.32e-5 & 1.64e-4 & 3.05e-5 & 1.37e-4 \\
\hline
\end{tabular}
\end{table}

\begin{table}[H]
\centering
\caption{Integration based approach: varying $h$ with $\tau = h$}
\label{tab:integration-h}
\begin{tabular}{c|c|c|c|c}
\hline
$h$ & $Er(u)$ & $Er(p)$ & $E^{(2)}(u)$ & $E^{(2)}(p)$ \\
\hline
1/100 & 6.65e-3 & 2.44e-2 & 4.70e-3 & 2.01e-2 \\
1/200 & 1.66e-3 & 1.32e-2 & 1.18e-3 & 1.08e-2 \\
1/400 & 4.15e-4 & 9.03e-3 & 2.95e-4 & 7.35e-3 \\
1/800 & 1.04e-4 & 5.53e-3 & 7.38e-5 & 4.49e-3 \\
\hline
\end{tabular}
\end{table}

\subsection{Results for the Newton--Raphson based approach}
The Newton-Raphson approach exhibits significantly higher accuracy. Tables~\ref{tab:newton-tau} and~\ref{tab:newton-h} show near machine-precision error for $u(x,t)$ and consistently low error in $p(t)$.

\begin{table}[H]
\centering
\caption{Newton--Raphson based approach: varying $\tau$ for fixed $h = 1/100$}
\label{tab:newton-tau}
\begin{tabular}{c|c|c|c|c}
\hline
$\tau$ & $Er(u)$ & $Er(p)$ & $E^{(2)}(u)$ & $E^{(2)}(p)$ \\
\hline
1/200  & 2.09e-10 & 3.31e-3 & 1.48e-10 & 3.31e-3 \\
1/400  & 4.26e-11 & 2.48e-3 & 3.01e-11 & 2.48e-3 \\
1/800  & 5.67e-12 & 1.81e-3 & 4.01e-12 & 1.81e-3 \\
1/1600 & 3.64e-13 & 8.11e-4 & 2.59e-13 & 1.31e-4 \\
\hline
\end{tabular}
\end{table}

\begin{table}[H]
\centering
\caption{Newton--Raphson based approach: varying $h$ with $\tau = h$}
\label{tab:newton-h}
\begin{tabular}{c|c|c|c|c}
\hline
$h$ & $Er(u)$ & $Er(p)$ & $E^{(2)}(u)$ & $E^{(2)}(p)$ \\
\hline
1/100 & 3.06e-9 & 5.80e-3 & 2.16e-9 & 5.80e-3 \\
1/200 & 2.09e-10 & 2.70e-3 & 1.48e-10 & 2.70e-3 \\
1/400 & 4.25e-11 & 1.76e-3 & 3.01e-11 & 1.76e-3 \\
1/800 & 6.17e-12 & 1.03e-3 & 4.21e-12 & 1.03e-3 \\
\hline
\end{tabular}
\end{table}

\subsection{Results for the PINN method}
The PINN approach was trained using 1000 interior collocation points. Table~\ref{tab:comparison} presents a comparative summary of all three methods.

\begin{table}[H]
\centering
\caption{Comparison of the three methods ($h = \tau = 0.01$, $N_f = 1000$ for PINN)}
\label{tab:comparison}
\begin{tabular}{l|c|c|c|c}
\hline
Approach & $Er(u)$ & $Er(p)$ & $E^{(2)}(u)$ & $E^{(2)}(p)$ \\
\hline
Integration     & 6.65e-3 & 2.44e-2 & 4.70e-3 & 2.01e-2 \\
Newton--Raphson & 3.06e-9 & 5.80e-3 & 2.16e-9 & 5.80e-3 \\
PINN (1000 pts) & 1.78e-4 & 1.29e-2 & 8.45e-5 & 3.69e-3 \\
\hline
\end{tabular}
\end{table}

\subsection{Discussion}
The results clearly show that the Newton--Raphson method delivers the highest precision for both state and parameter recovery. The integration method is simpler and robust, but its accuracy is limited due to numerical differentiation of \( g(t) \). The PINN approach offers a flexible and mesh-free alternative, producing continuous approximations and performing reasonably well with a limited number of training points.

\begin{remark}
In all cases, the convergence behavior of the schemes aligns with the theoretical order of accuracy, and the residual errors confirm both stability and robustness of the proposed methods.
\end{remark}

\subsection{Graphical comparisons and interpretation}

Figure~\ref{fig1} compares the analytical solution and the numerical results for \( u(x,t) \) and \( p(t) \) obtained using \textbf{Integration-based approach} on a uniform grid with \( N = M = 200 \).

\begin{figure}[H]
\centering
\includegraphics[width=0.7\textwidth]{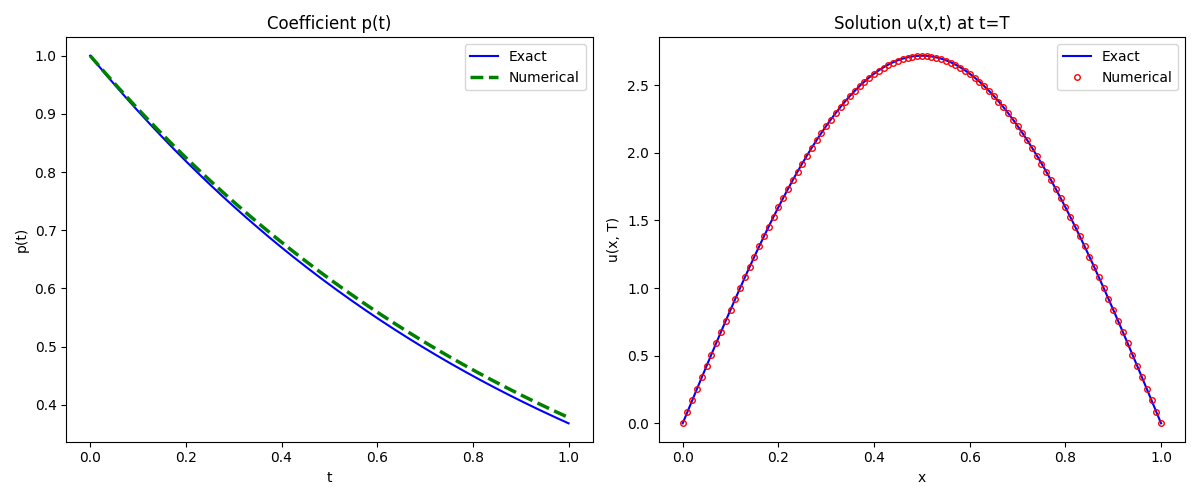}
\caption{Analytical vs numerical solutions of \( p(t) \) and \( u(x, t) \) at \( T=1 \) using the Integration-based approach.}
\label{fig1}
\end{figure}

Figure~\ref{fig2} presents the corresponding results for \textbf{Newton--Raphson-based approach}.

\begin{figure}[H]
\centering
\includegraphics[width=0.7\textwidth]{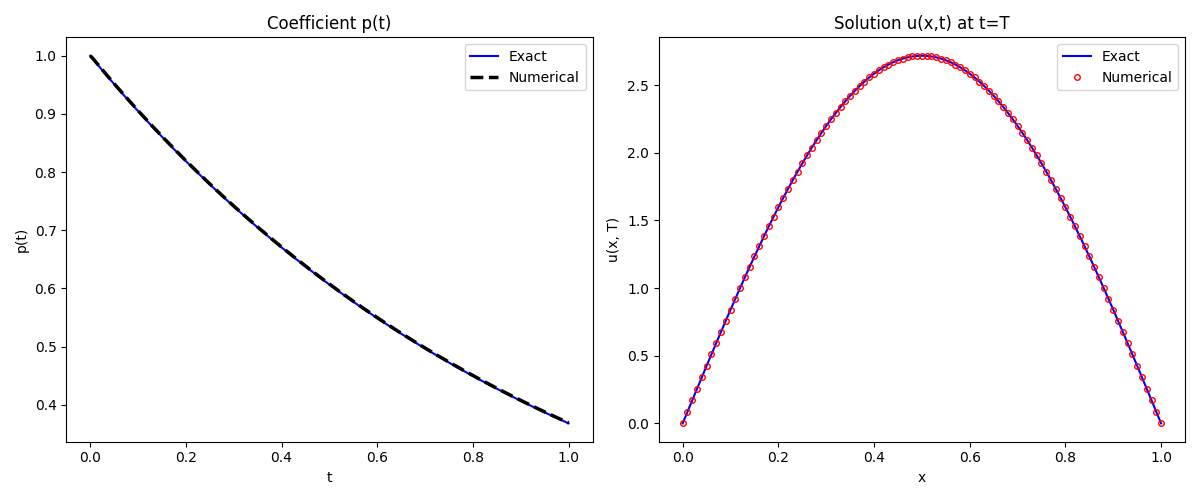}
\caption{Analytical vs numerical solutions of \( p(t) \) and \( u(x, t) \) at \( T=1 \) using the Newton--Raphson-based approach.}
\label{fig2}
\end{figure}

Figure~\ref{fig3} shows the comparison for \textbf{PINN approach} under the same discretisation.

\begin{figure}[H]
\centering
\includegraphics[width=0.7\textwidth]{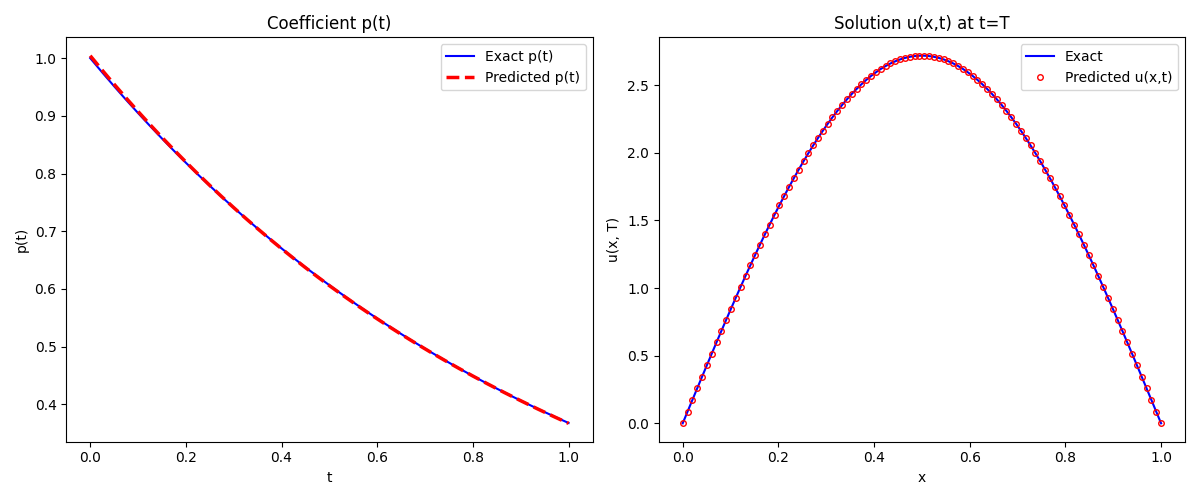}
\caption{Analytical vs numerical solutions of \( p(t) \) and \( u(x, t) \) at \( T=1 \) using PINN approach.}
\label{fig3}
\end{figure}

A direct comparison of the reconstructed coefficient \( p(t) \) across all three approaches is shown in Figure~\ref{fig4}, highlighting their relative accuracy.

\begin{figure}[H]
\centering
\includegraphics[width=0.8\textwidth]{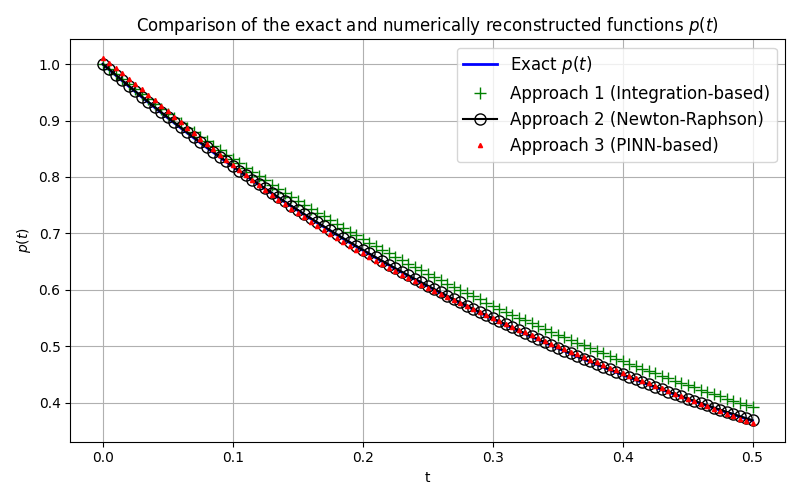}
\caption{Comparison of reconstructed \( p(t) \) from all three approaches against the exact solution.}
\label{fig4}
\end{figure}

From Figures~\ref{fig1}--\ref{fig3}, it is evident that \textbf{Newton--Raphson approach} achieves superior accuracy in both state and coefficient reconstruction. Notably, it retains high fidelity even when the spatial step is coarser than the temporal step. Precision can be further improved with a stricter convergence tolerance in the iteration.

Figure~\ref{fig5} displays 2D plots of the exact and computed \( u(x,t) \) using the Newton--Raphson method on a grid with \( N=M=200 \).

\begin{figure}[H]
\centering
\includegraphics[width=0.7\textwidth]{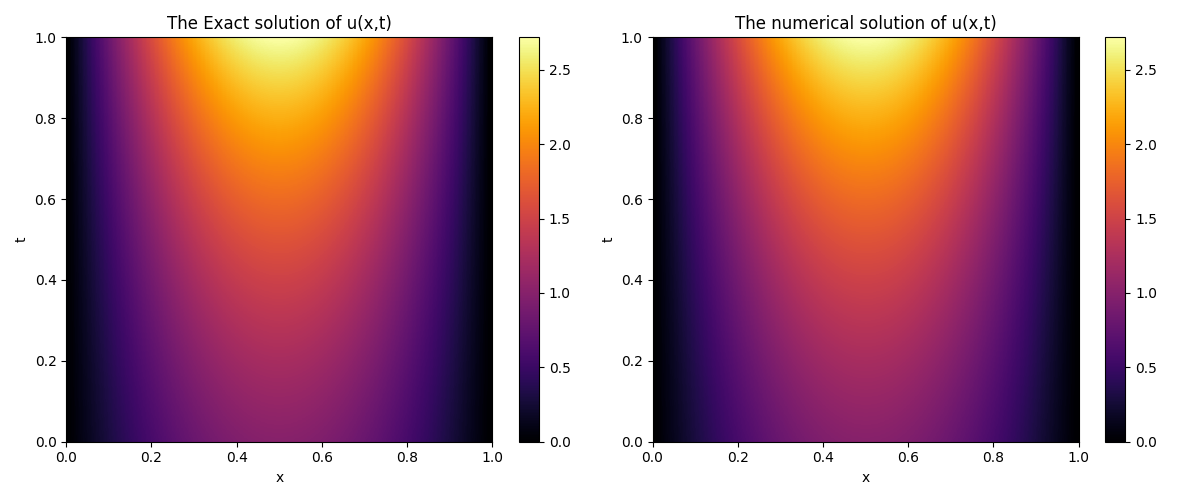}
\caption{Exact (left) and numerical (right) solutions of \( u(x,t) \) using the Newton-Raphson-based approach.}
\label{fig5}
\end{figure}

These visual results confirm the stability and effectiveness of the proposed schemes. The Newton-Raphson based approach demonstrates excellent performance, combining rapid convergence and high precision. The integration method is computationally simple but less accurate, particularly in estimating \( p(t) \). The PINN method offers a flexible and mesh-free framework, yielding continuous approximations, but with higher training cost and sensitivity to network architecture and hyperparameters.

\subsection{Sensitivity of the numerical methods to noisy data}
To test our methods under realistic conditions, we conducted experiments in which both the overdetermination data \(g(t)\) and its time derivative were corrupted by additive noise. We defined
\[
g^{\delta}(t) = g(t) + \delta\,\eta(t),
\qquad
\partial_t g^{\delta}(t) = \partial_t g(t) + \delta\,\xi(t),
\]
with \(\delta \in \{0.01,0.03,0.05\}\) (corresponding to 1 \%, 3 \%, and 5 \% noise levels) and \(\eta(t)\), \(\xi(t)\) as random perturbations simulating measurement errors. The noisy signals \(g^{\delta}(t)\) and \(\partial_t g^{\delta}(t)\) were then used in place of the exact values in the coefficient‐recovery formula.

\textbf{Testing the Integration-based approach on noisy data.}
The results are presented in Figures~\ref{noise1}, where each plot compares the exact and numerically reconstructed values of \( p(t) \) under different noise levels.
\begin{figure}[h]
       \centering
        \includegraphics[width=0.6 \textwidth]{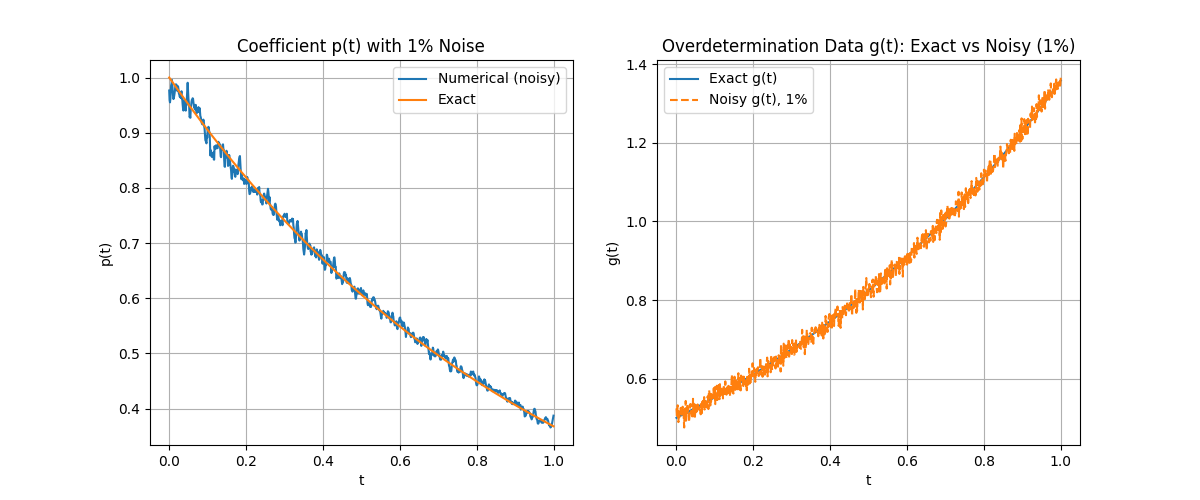}
    \vspace{0.3cm} 
        \centering
        \includegraphics[width=0.6 \textwidth]{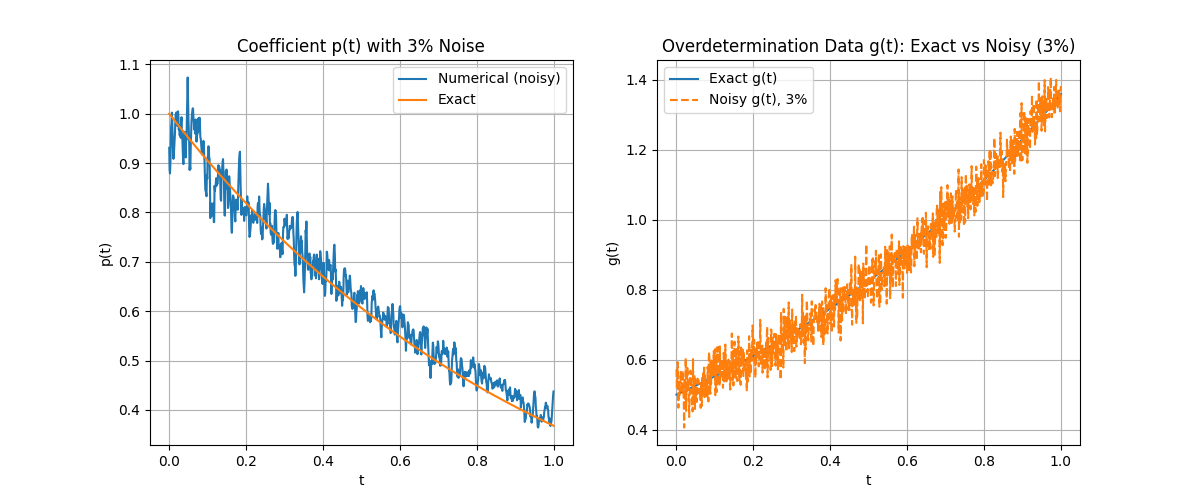}
     \vspace{0.3cm}
        \centering
        \includegraphics[width=0.6 \textwidth]{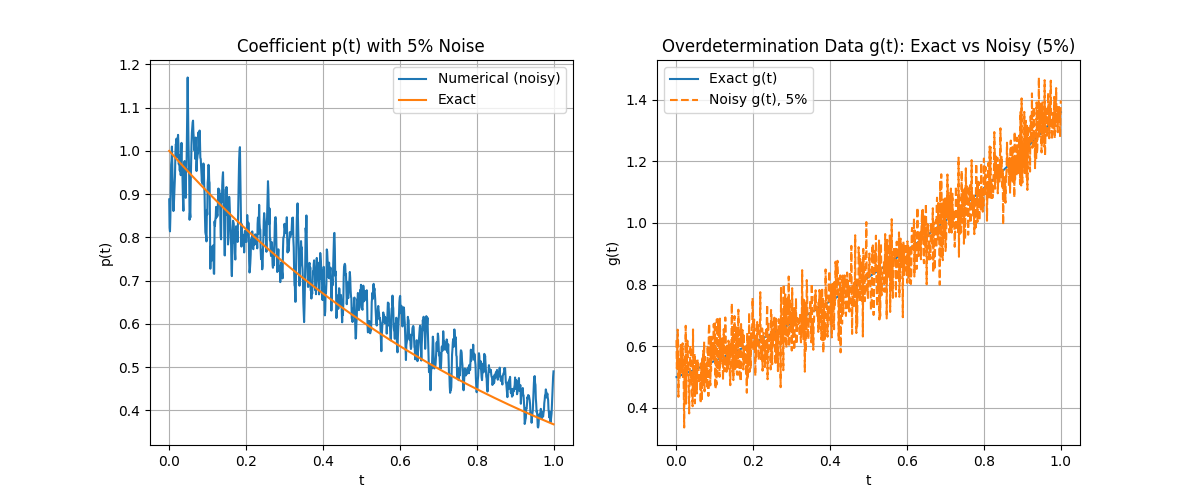}
    \caption{Left: Exact and numerically reconstructed \(p(t)\) via the integration-based method under 1 \%, 3 \%, and 5 \% noise in the overdetermination data. Right: Overdetermination data \(g(t)\) showing the exact signal (blue) and its noisy observations (orange).}
    \label{noise1}
\end{figure}
The results demonstrate the ability of the integration-based approach to efficiently handle noisy data, making it suitable for practical applications where accurate measurements are rarely available.

\textbf{Testing the Newton–Raphson-based approach on noisy data.}
Figure \ref{noise2} presents the comparison results of the exact and numerically reconstructed values of \( p(t) \) under 1\% noise level obtained using the Newton–Raphson–based approach.
\begin{figure}[h]
       \centering
        \includegraphics[width=0.55 \textwidth]{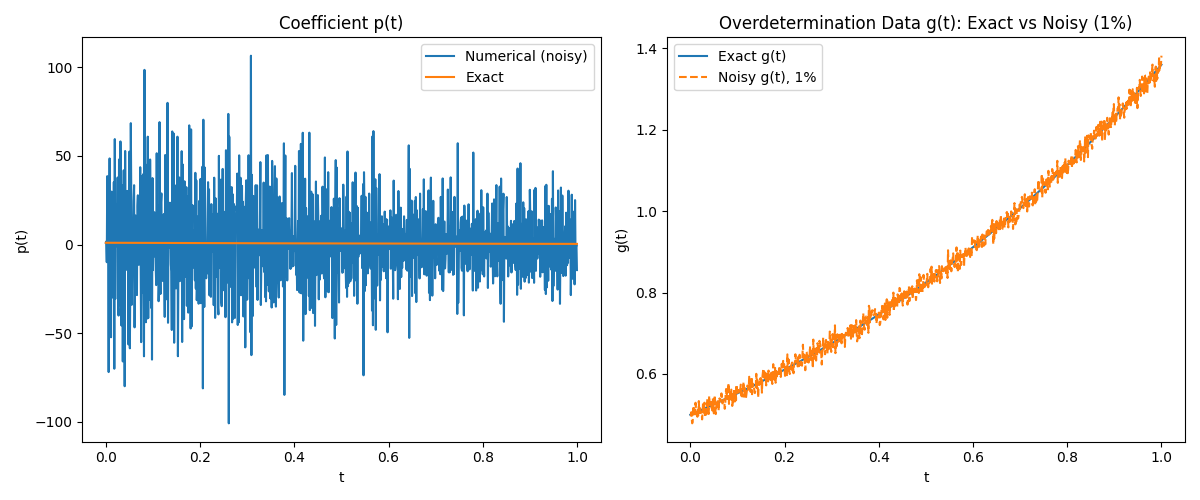}
\caption{Left: Exact and numerically reconstructed \(p(t)\) via the Newton–Raphson-based approach under 1 \% noise in the overdetermination data. Right: Overdetermination data \(g(t)\) showing the exact signal (blue) and its noisy observations (orange).}
\label{noise2}
\end{figure}

From the figure \ref{noise2}, the Newton–Raphson approach proves unstable even at 1\% noise, reflecting the \emph{ill-posed} nature of the inverse problem: tiny data errors cause large parameter deviations. Despite Savitzky–Golay filtering and Tikhonov regularisation, improvements are marginal, underscoring the method’s limited robustness and the need for more noise-tolerant inversion strategies.

\textbf{Testing the PINN approach on noisy data}
Figure \ref{noise3} shows the comparison of the exact and predicted values of \( p(t) \) under 1\%, 3\%, 5\% noise levels.
\begin{figure}[h]
       \centering
        \includegraphics[width=0.5 \textwidth]{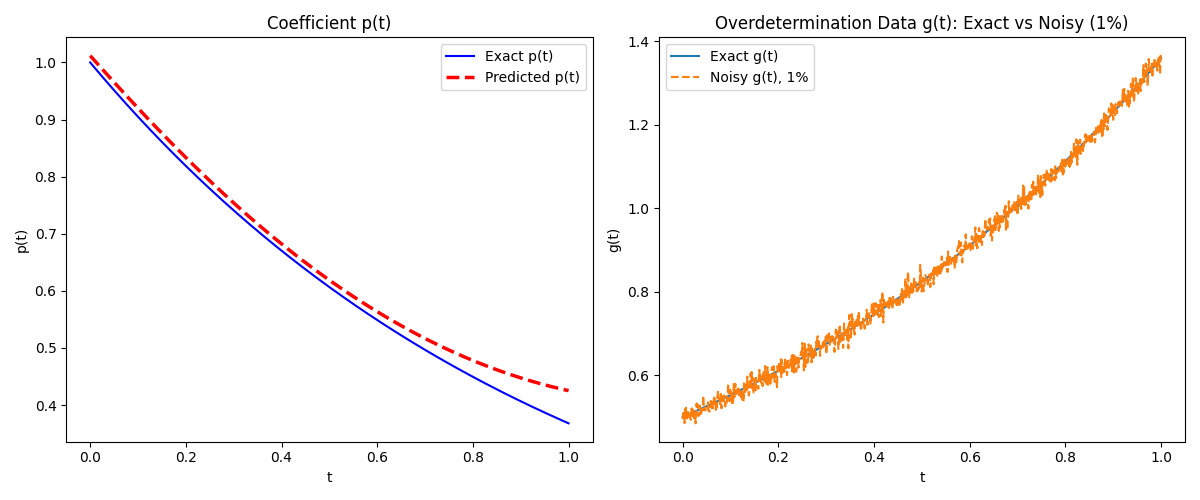}
    \vspace{0.3cm} 
        \centering
        \includegraphics[width=0.5 \textwidth]{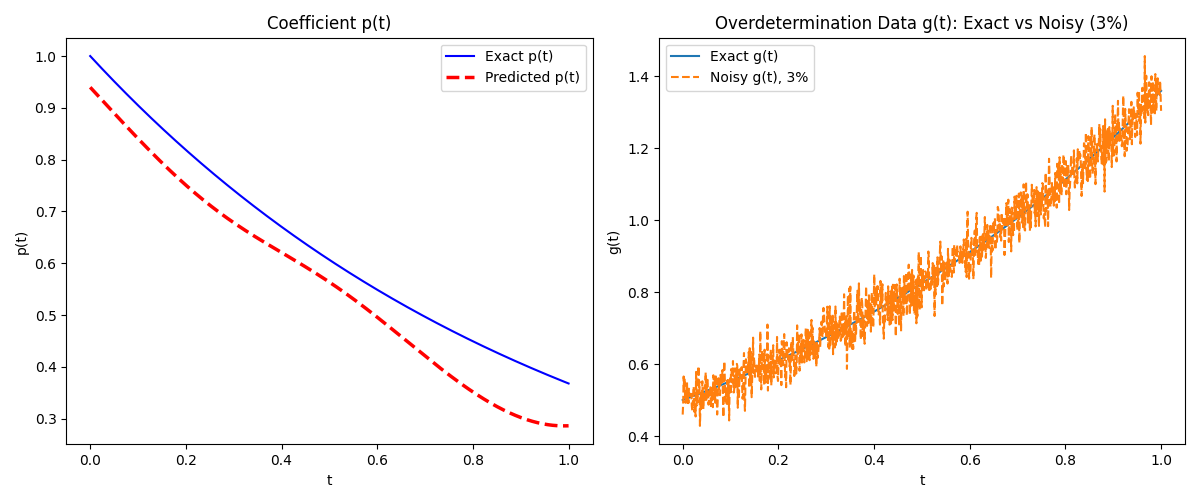}
     \vspace{0.3cm}
        \centering
        \includegraphics[width=0.5 \textwidth]{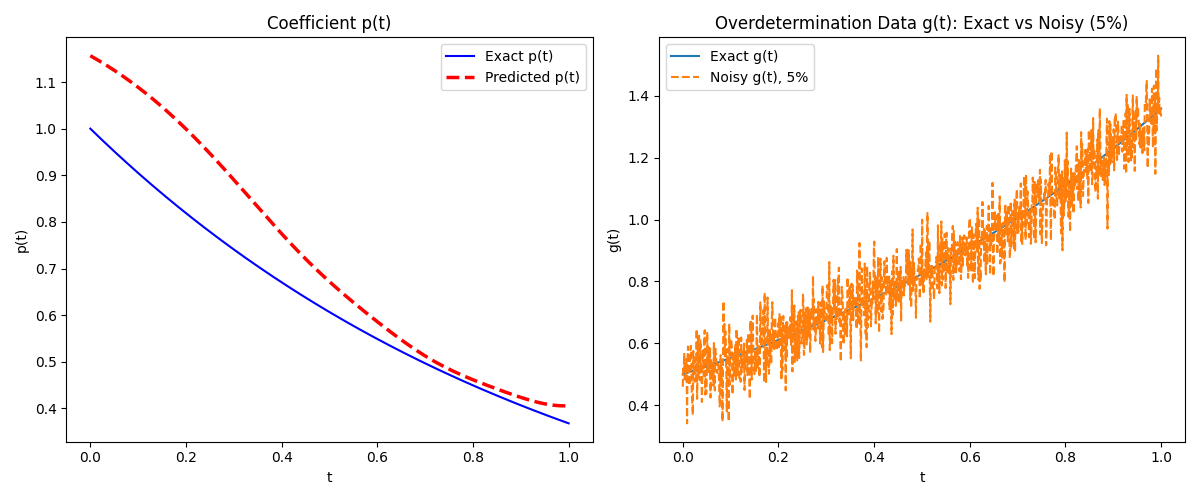}
    \caption{Left: Exact and numerically reconstructed \(p(t)\) via the PINN method under 1 \%, 3 \%, and 5 \% noise in the overdetermination data. Right: Overdetermination data \(g(t)\) showing the exact signal (blue) and its noisy observations (orange).}
    \label{noise3}
\end{figure}

From the figure \ref{noise3}, it is clear that the PINN approach remains stable even with 1\%–5\% noise and yields significantly smoother reconstructions than the integration-based method. By embedding the diffusion equation directly into the loss function, PINNs automatically suppress noise-induced oscillations and enforce physical consistency. As a result, the recovered potential $p(t)$ not only matches the noisy measurements but also adheres closely to the true dynamics, delivering accurate, smooth solutions across all tested noise levels.

\section{Conclusion}

We studied the inverse problem of identifying a time‐dependent potential \(p(t)\) and the state \(u(x,t)\) in a 1D diffusion equation with Dirichlet boundary conditions and an integral overdetermination constraint. After establishing existence and uniqueness via Schauder’s theorem, we implemented three numerical approaches:

\begin{itemize}
  \item \textbf{Integration-based:} Simple to implement, stable under 1–5 \% noise, but sensitive to grid resolution and derivative approximation.
  \item \textbf{Newton–Raphson:} Fastest and most accurate on noise-free data, yet unstable even at 1 \% noise.
  \item \textbf{PINNs:} Mesh-free and physics-regularised, providing smooth, noise-robust reconstructions (up to 5 \% noise) at the cost of longer training and careful hyperparameter tuning.
\end{itemize}

Overall, classical methods remain highly efficient for low-dimensional, low-noise settings, while PINNs offer a flexible, robust alternative for noisy or higher-dimensional problems. Future work will explore hybrid schemes that combine the speed and accuracy of classical solvers with the noise resilience and mesh independence of PINNs.

\section*{Funding}
\noindent
The research is financially supported by a grant from the
Ministry of Science and Higher Education of the Republic of Kazakhstan (No. AP27508473), by the FWO Research Grant G083525N: Evolutionary partial differential equations with strong singularities, and by the Methusalem programme of the Ghent University Special Research Fund (BOF) (Grant number 01M01021).


\end{document}